\documentclass[a4,12pt]{amsart}
\usepackage{amssymb,amstext,amsmath,amscd,amsthm,amsfonts,enumerate,graphicx,latexsym,color}
\usepackage[all]{xy}
\usepackage{cite}
\makeatletter
\@addtoreset{equation}{section}

\newtheorem{thm}{Theorem}[section]
\newtheorem*{thm*}{Theorem}

\newtheorem{cor}[thm]{Corollary}
\newtheorem{lem}[thm]{Lemma}
\newtheorem{prop}[thm]{Proposition}
\theoremstyle{definition}
\newtheorem{dfn}[thm]{Definition}
\newtheorem*{dfn*}{Definition}
\newtheorem{rem}[thm]{Remark}

\newtheorem{ex}[thm]{Example}

\newtheorem*{nota*}{Notation}
\theoremstyle{remark}

\newtheorem*{ac}{Acknowledgments}

\newtheorem*{claim*}{Claim}
\renewcommand{\qedsymbol}{$\blacksquare$}
\numberwithin{equation}{thm}
\def\p{\mathfrak{p}}
\def\q{\mathfrak{q}}

\def\P{\mathcal{P}}
\def\Q{\mathcal{Q}}

\def\dm{\operatorname{\mathsf{D^{\mbox{\boldmath$-$}}}}}
\def\db{\operatorname{\mathsf{D^{\mathsf{b}}}}}
\def\ltensor{\otimes^{\mathbf{L}}}

\def\mod{\operatorname{\mathsf{mod}}}
\def\Mod{\operatorname{\mathsf{Mod}}}
\def\spc{\operatorname{\mathtt{Spc}}}
\def\spec{\operatorname{\mathsf{Spec}}}
\def\supp{\operatorname{\mathsf{Supp}}}
\def\spp{\operatorname{\mathsf{Spp}}}
\def\min{\mathsf{Min}}
\def\max{\mathsf{Max}}
\def\V{\mathsf{V}}
\def\H{\mathsf{H}}
\def\ann{\operatorname{\mathsf{Ann}}}

\def\h{\operatorname{\mathsf{ht}}}

\def\one{\mathbf{1}}
\def\pp{\mathfrak{s}}
\def\s{\mathcal{S}}
\def\T{\mathcal{T}}
\def\u{\mathtt{U}}
\def\X{\mathcal{X}}

\title[Connectedness of the Balmer spectra]{Connectedness of the Balmer spectra of right bounded derived categories}
\author{Hiroki Matsui} 
\address{Graduate School of Mathematics, Nagoya University, Furocho, Chikusaku, Nagoya, Aichi 464-8602, Japan}
\email{m14037f@math.nagoya-u.ac.jp}
\date{\today}
\thanks{The author is supported by Grant-in-Aid for JSPS Fellows 16J01067.
}
\begin{document} 

\begin{abstract}
By virtue of Balmer's celebrated theorem, the classification of thick tensor ideals of a tensor triangulated category $\T$ is equivalent to the topological structure of its Balmer spectrum $\spc \T$.
Motivated by this theorem, we discuss connectedness and noetherianity of the Balmer spectrum of  a right bounded derived category of finitely generated modules over a commutative ring.
\end{abstract}

\maketitle

\section{Introduction}
Tensor triangulated geometry is a theory introduced by Balmer \cite{B} to study tensor triangulated categories by algebro-geometric methods.
Let $(\T, \otimes, \one)$ be an essentially small tensor triangulated category (i.e., a triangulated category $\T$ equipped with a symmetric monoidal tensor product $\otimes$ which is compatible with the triangulated structure).
Then Balmer defined a topological space $\spc \T$ which we call the Balmer spectrum of $\T$.
A celebrated theorem due to Balmer \cite{B} states that the radical thick tensor ideals of $\T$ are classified by using the geometry of $\spc \T$:

\begin{thm}[Balmer]\label{B}
There is an order-preserving one-to-one correspondence 
$$
\xymatrix{
\{ \mbox{radical thick tensor ideals of } \T\}  \ar@<0.5ex>[r]^-f 
& \{\mbox{Thomason subsets of } \spc \T\}, \ar@<0.5ex>[l]^-g
}
$$
where $f$ and $g$ are given by $f(\X):=\spp \X:=\bigcup_{X \in \X} \spp X$ and $g(W) := \spp^{-1}(W):=\{X \in \T \mid \spp X \subseteq W \}$, respectively.
\end{thm}

From this result, if we want to classify the radical thick tensor ideals of a given tensor triangulated category $\T$, we have only to understand the topological space $\spc \T$.
Therefore, it is crucial to discuss topological properties of the Balmer spectrum.

In this paper, we consider the right bounded derived category $\dm(\mod R)$ of a commutative noetherian ring $R$.
This triangulated category is a tensor triangulated category with respect to derived tensor products, and we can consider its Balmer spectrum $\spc \dm(\mod R)$.
The main results of this paper are the following two theorems:

\begin{thm}[Theorem \ref{noeth}]
If the Balmer spectrum $\spc \dm(\mod R)$ is a noetherian topological space, then the Zariski spectrum $\spec R$ is a finite set.
\end{thm}

\begin{thm}[Corollary \ref{conn}]
The Balmer spectrum $\spc \dm(\mod R)$ is connected if and only if the Zariski spectrum $\spec R$ is so.
\end{thm}
\noindent Moreover, by using the latter theorem, we give a variant of a well-known result of Carlson \cite{C} in representation theory.

This paper is organized as follows.
In Section 2, we recall some basic materials from tensor triangulated geometry and topology theory. 
In Section 3, we prove our main theorems and give applications.
\section{Preliminaries}
Throughout this paper, let $R$ be a commutative noetherian ring.
For an ideal $I$ of $R$, we denote by $\V(I)$ the set of prime ideals of $R$ containing $I$.
We note that $\V(I)$ is a closed subset of the Zariski spectrum $\spec R$ and $\V(\p)$ is the closure $\overline{\{\p\}}$ of $\p$ in $\spec R$.
Denote by $\dm(R)$ (resp. $\db(R)$) the derived category of complexes $X$ of finitely generated $R$-modules with $\H^i(X)=0$ for all $i \gg 0$ (resp. $|i| \gg 0$).
Then $\dm(R)$ is an essentially small tensor triangulated category via derived tensor product $\ltensor_R$ with unit $R$.

First we will recall the definitions of a thick tensor ideal, a radical thick tensor ideal, and a prime thick tensor ideal.

\begin{dfn}
Let $\T$ be an essentially small tensor triangulated category.
\begin{enumerate}
\item
A subcategory $\X$ of $\T$ is called a {\it thick tensor ideal} of $\T$ if it is a thick subcategory of $\T$ and for any $T \in \T$ and $X \in \X$, the tensor product $T \otimes X$ belongs to $\X$.
\item
For a thick tensor ideal $\X$ of $\T$, we denote by $\sqrt{\X}$ the {\it radical} of $\X$, that is, the subcategory of $\T$ consisting of objects $X$ such that the $n$-fold tensor product $X \otimes X \otimes \cdots \otimes X$ belongs to $\X$ for some integer $n \ge 1$.
\item
A thick tensor ideal $\X$ of $\T$ is called {\it radical} if $\sqrt{\X}=\X$.
\item
A proper thick tensor ideal $\P$ of $\T$ is called {\em prime} if $X\otimes Y$ is in $\P$, then so is either $X$ or $Y$.
The set of prime thick tensor ideals of $\T$ is denoted by $\spc\T$ and we call it the {\em Balmer spectrum} of $\T$.
\end{enumerate}
\end{dfn}

For a thick tensor ideal $\X$, its radical $\sqrt{\X}$ is a thick tensor ideal.
Indeed, by \cite[Lemma 4.2]{B}, it is equal to the intersection of all prime thick tensor ideals containing $\X$. 

\begin{ex}\label{ex1}
For a complex $X \in \dm(R)$, define the {\it support} of $X$ by
\begin{align*}
\supp X &:=\{\p \in \spec R \mid X_\p \not\cong 0 \mbox{ in } \dm(R_\p)\} \\
&= \bigcup_{ n \in \mathbb{Z}} \supp \H^n(X).
\end{align*}
Moreover, for a class $\X$ of objects of $\dm(R)$, denote by $\supp \X := \bigcup_{X \in \X}\supp X$.
Then, for a subset $W$ of $\spec R$, 
$$
\supp^{-1}W:=\{X \in \dm(R) \mid \supp X \subseteq W\}
$$
is a thick tensor ideal of $\dm(R)$.
Furthermore, if we take $W:=\{ \q \in \spec R \mid \q \not\subseteq \p \}$ for a fixed $\p$, then
$$
\s(\p) := \supp^{-1} W =\{X \in \dm(R) \mid X_\p \cong 0 \mbox{ in } \dm(R_\p)\}
$$
is a prime thick tensor ideal of $\dm(R)$.
For the proof, please see \cite{MT} for instance.
\end{ex}

Balmer \cite{B} defined a topology on $\spc \T$ as follows.

\begin{dfn}
\begin{enumerate}
\item
For an object $X \in \T$, the {\it Balmer support} of $X$, denoted by $\spp X$, is defined as the set of prime thick tensor ideals not containing $X$. 
Set $\u(X):= \spc \T \setminus \spp X$.
\item
Define a topology on $\spc \T$ whose open basis is $\{\u(X) \mid X \in \T \}$.
\end{enumerate}
\end{dfn}
We always consider this topology on the Balmer spectrum.



Next, let us recall some notions from topology theory for later use.

\begin{dfn}
Let $X$ be a topological space.
\begin{enumerate}
\item
We say that a subspace of $X$ is a {\it clopen subset} if it is closed and open in $X$.
\item
A subspace $W$ of $X$ is said to be {\it specialization closed} if for any $x \in W$, $\overline{\{x\}} \subseteq W$ holds.
\item
A subspace $W$ of $X$ is said to be {\it generalization closed} if for any $x \in W$ and $y \in X$, $x \in \overline{\{y\}}$ implies $y \in W$.
\item
We say that $X$ is {\it connected} if it contains no non-trivial clopen subset.
For a subspace $Y$ of $X$, we say that $Y$ is a {\it connected subspace} of $X$ if it is a connected space by induced topology. 
Moreover, a {\it connected component} of $X$ is a maximal connected subspace of $X$.
\item
We say that $X$ is {\it irreducible} if it can not be the union of two proper closed subspaces.
For a subspace $Y$ of $X$, we say that $Y$ is an {\it irreducible subspace} of $X$ if it is an irreducible space by induced topology. 
Moreover, an {\it irreducible component} of $X$ is a maximal irreducible subspace of $X$, which is automatically closed since the closure of irreducible subspace is also irreducible.
\item
We say that $X$ is {\it noetherian} if every descending chain of closed subspaces stabilizes.
\end{enumerate}
\end{dfn}

\begin{rem}
\begin{enumerate}[\rm(1)]
\item
A subspace is generalization closed if and only if its complement is specialization closed.
\item
Let $X \supseteq Y \supseteq Z$ be subspaces.
If $Y$ is specialization closed in $X$ and $Z$ is specialization closed in $Y$, then $Z$ is specialization closed in $X$.
\item
Let $W$ be a subspace of $\spec R$.
Then $W$ is specialization closed (resp. generalization closed) in $\spec R$ if and only if 
\begin{align*}
\p \in W,\,\, \p \subseteq \q \Longrightarrow \q \in W. \\
(\mbox{resp. }\q \in W,\,\, \p \subseteq \q \Longrightarrow \p \in W.)
\end{align*}
\item (Balmer \cite{B})
Let $\T$ be an essentially small tensor triangulated category and $W$ a subspace of $\spc \T$.
Then $W$ is specialization closed (resp. generalization closed) in $\spc \T$ if and only if 
\begin{align*}
\P \in W,\,\, \P \supseteq \Q \Longrightarrow \Q \in W. \\
(\mbox{resp. } \Q\in W,\,\, \P \supseteq \Q \Longrightarrow \P \in W.)
\end{align*}
\end{enumerate}
\end{rem}

\begin{lem}\label{6}
Let $X$ be a topological space. 
Then every connected component of $X$ is both specialization closed and generalization closed.
\end{lem}

\begin{proof}
Fix a connected component $O$ of $X$.
For $x \in O$, $\overline{\{x\}}$ is irreducible and in particular connected.
Since $O \cap \overline{\{x\}}$ is non-empty, $O \cup \overline{\{x\}}$ is connected.
Thus, $O \cup \overline{\{x\}}$ must be equal to $O$, and hence $\overline{\{x\}} \subseteq O$.
This shows that $O$ is specialization closed in $X$.

For $x \not\in O$, assume that there exists $y \in \overline{\{x\}}$ with $y \in O$.
Then $\overline{\{x\}} \cap O$ is non-empty as it contains $y$.
Therefore, the same argument as above shows that $\overline{\{x\}} \subseteq O$.
This gives a contradiction to $x \notin O$.
Thus, $X \setminus O$ is specialization closed in $X$ and hence $O$ is generalization closed in X.
\end{proof}

\section{Main theorems}
In this section, we discuss noetherianity, connectedness, and irreducibility of the Balmer spectrum $\spc \dm(R)$.

\subsection{Noetherianity}
Besides, we show the following theorem which gives a sufficient condition for noetherianity of the Balmer spectrum $\spc \dm(R)$.

\begin{thm}\label{noeth}
If the Balmer spectrum $\spc \dm(R)$ is a noetherian topological space, then $\spec R$ is a finite set (i.e., semi-local ring with Krull dimension at most $1$).
\end{thm}

Before proving this, we give the following easy lemma.

\begin{lem}
If $\spec R$ has infinitely many prime ideals, then there is a countable set of prime ideals which have no inclusion relations among them. 
\end{lem}

\begin{proof}
If $R$ has infinitely many maximal ideals, then we can take such a set as a countable set of maximal ideals.

Assume that $R$ has only finitely many maximal ideals.
Then $R$ has finite Krull dimension.
Since $R$ has infinitely many prime ideals, there is a non-negative integer $n$ such that the set $\{\p \in \spec R \mid \h \p = n \}$ has infinitely many elements.
Thus, a countable subset of this set satisfies the property what we want.
\end{proof}

For a complex $X \in \dm(R)$, denote by $\langle X \rangle$ the smallest thick tensor ideal of $\dm(R)$ containing $X$.

\begin{proof}[Proof of Theorem \ref{noeth}]
Assume that $\spc \dm(R)$ is noetherian.
Since any descending chain of closed subsets of $\spc \dm(R)$ stabilizes, in particular, any descending chain $\spp X_1 \supseteq \spp X_2 \supseteq \spp X_3 \cdots$ stabilizes.
Then by using Theorem \ref{B}, every descending chain $\sqrt{\langle X_1 \rangle} \supseteq \sqrt{\langle X_2 \rangle} \supseteq \sqrt{\langle X_3 \rangle} \supseteq \cdots$ stabilizes.
Indeed, the one-to-one correspondence in the theorem assigns $\spp X$ to $\sqrt{\langle X \rangle}$.

Assume furthermore that $R$ has infinitely many prime ideals.
From the previous lemma, we can take countably many prime ideals $\{\p_n\}_{n \ge 1}$ which have no inclusion relation to each other.
Set $X_n := \bigoplus_{i \ge n}R/\p_i[i]$ to be the complex
$$
X_n:=(\cdots  \xrightarrow{0} R/\p_{n+2} \xrightarrow{0} R/\p_{n+1} \xrightarrow{0} R/\p_{n} \to 0 \cdots).
$$
Here, $R/\p_i$ fit into the $i$-th component.
Then $X_n$ belongs to $\dm(R)$, and $X_{n+1}$ is a direct summand of $X_n$ for each integer $n \ge 1$.
Therefore, we have a descending chain $\sqrt{\langle X_1 \rangle} \supseteq \sqrt{\langle X_2 \rangle} \supseteq \sqrt{\langle X_3 \rangle} \supseteq \cdots$ of radical thick tensor ideals.
From the above argument, we get an equality $\sqrt{\langle X_n \rangle} = \sqrt{\langle X_{n+1} \rangle}$ for some integer $n\ge 1$.
Taking $\supp$, we obtain 
$$
\bigcup_{i \ge n}\V(\p_i) = \supp \sqrt{\langle X_n \rangle} = \supp \sqrt{\langle X_{n+1} \rangle} = \bigcup_{i \ge n+1}\V(\p_i).
$$
Hence, there is an integer  $m \ge n+1$ such that $\p_m \subseteq \p_n$.
This gives a contradiction.
\end{proof}


\begin{rem}
If $R$ is artinian, then by \cite[Theorem 6.5]{MT}, $\spc \dm(R)$ is homeomorphic to $\spec R$.
In particular, $\spc \dm(R)$ is a noetherian topological space.
\end{rem}

\subsection{Connectedness}
In this subsection, we mainly discuss connectedness of the Balmer spectrum $\spc \dm(R)$.
We use the following pair of maps defined in \cite{MT} to compare two spectra:
$$
\pp: \spc \dm(R) \rightleftarrows \spec R :\s.
$$
Here the map $\s$ was defined in Example \ref{ex1}.
Let me list some basic properties of these maps in the following proposition. 
However we don't give the definition of the map $\pp$, the last statement of the following proposition gives the characterization of this map.

\begin{prop}{\rm{\cite[Theorem 3.9, Corollary 3.10, Theorem 4.5]{MT}}}{\label{maps}}
\begin{enumerate}[\rm(1)]
\item
Both maps $\pp$ and $\s$ are order-reversing.
\item
$\pp$ is continuous.
\item
$\pp \cdot \s = 1$.
In particular, $\pp$ is surjective and $\s$ is injective.
\item
For a prime thick tensor ideal $\P$ of $\dm(R)$, one has 
$$
\s\pp(\P)=\supp^{-1} \supp(\P) \supseteq \P. 
$$
\item
For a prime thick tensor ideal $\P$ of $\dm(R)$, one has 
$$
\supp \P = \{\p \in \spec R \mid \p \not\subseteq \pp(\P)\}.
$$
\end{enumerate}
\end{prop}

\begin{rem}
As it has been shown in \cite[Theorem 4.7]{MT}, $\s$ is not continuous in general.
\end{rem}

The following theorem is the main result of this subsection.

\begin{thm}\label{main}
Let $C \in \db(R)$ be a bounded complex. 
\begin{enumerate}[\rm(1)]
\item
There is a one-to-one correspondence
$$
\xymatrix{
\left\{ 
\begin{matrix}
\text{connected components of $\spp C$}
\end{matrix}
\right\}  
\ar@<0.5ex>[r]^-\pp &
\ar@<0.5ex>[l]^-{\pp^{-1}}
}
\!\!\!
\left\{ 
\begin{matrix}
\text{connected components of $\supp C$}
\end{matrix}
\right\}.
$$
\item
There is a one-to-one correspondence
$$
\xymatrix{
\left\{ 
\begin{matrix}
\text{irreducible components of $\spp C$}
\end{matrix}
\right\}  
\ar@<0.5ex>[r]^-\pp &
\ar@<0.5ex>[l]^-{\pp^{-1}}
}
\!\!\!
\left\{ 
\begin{matrix}
\text{irreducible components of $\supp C$}
\end{matrix}
\right\}.
$$
\end{enumerate}
\end{thm}

The proof of this theorem is divided into several lemmata.

Fix a bounded complex $C \in \db(R)$.
Then by \cite[Proposition 2.9]{MT}, for a thick tensor ideal $\X$, $C \in \X$ if and only if $\supp C \subseteq \supp \X$.
By combining this with Proposition \ref{maps}, $\P \in \spp C$ if and only if $\s\pp(\P) \in \spp C$ for a prime thick tensor ideal $\P$ of $\dm(R)$.

\begin{lem}
\begin{enumerate}[\rm(1)]
\item
$\pp(\spp C) = \supp C$.
\item
$\s(\supp C) \subseteq \spp C$.
\end{enumerate}
\end{lem}

\begin{proof}
$(1)$
For a prime thick tensor ideal $\P$ in $\spp C$, we have the following equivalences:
\begin{align*}
\P \in \pp^{-1}(\supp C) &\Leftrightarrow \pp(\P) \in \supp C \\
&\Leftrightarrow \supp C \not\subseteq \supp \P = \{\p \in \spec R \mid \p \not\subseteq  \pp(\P)\} \\
&\Leftrightarrow C \not\in \P \\
&\Leftrightarrow \P \in \spp C.
\end{align*}
Here, the first and the last equivalences are clear.
Since $\{\p \in \spec R \mid \p \not\subseteq  \pp(\P)\}$ is the largest specialization closed subset of $\spec R$ not containing $\pp(\p)$, the second equivalence holds.
The third one follows from the above discussion.
As a result, $\supp C = \pp(\pp^{-1}(\supp C)) = \pp(\spp C)$ since $\pp$ is surjective.

$(2)$
For an element $\p \in \supp C$, $\supp C \not\subseteq \supp \s(\p)=\{\q \in \spec R \mid \q \not\subseteq \pp\s(\p)=\p \}$ shows $C \not\in \s(\p)$.
Thus, we obtain $\s(\p) \in \spp C $.
\end{proof}

From this lemma, the maps
$$
\pp: \spc \dm(R) \rightleftarrows \spec R :\s 
$$
restrict to maps
$$
\pp: \spp C \rightleftarrows \supp C :\s 
$$

\begin{lem}\label{minmax}
The above pair of maps induce a one-to-one correspondence
$$
\pp: \max \spp C \rightleftarrows \min \supp C :\s. 
$$
Here, $\max \spp C$ (resp. $\min \supp C$) is the set of maximal (resp. minimal) elements of $\spp C$ (resp. $\supp C$) with respect to the inclusion relation. 
\end{lem}
If we take $C= R$, this lemma recovers \cite[Theorem 4.12]{MT}.

\begin{proof}
Because $\s : \spec R \to \spc \dm(R)$ is injective, we have only to check that the map $\s: \min \supp C \to \max \spp C$ is well-defined and surjective.
Let $\p$ be a minimal element of $\supp C$.
We show that $\s(\p)$ is a maximal element of $\spp C$.
Take a prime thick tensor ideal $\P$ in $\spp C$ containing $\s(\p)$.
Then $\pp(\P) \subseteq \pp\s(\p) = \p$ by Proposition \ref{maps}.
Since both $\p$ and $\pp(\P)$ belong to $\supp C$, the minimality of $\p$ shows the equality $\p= \pp(\P)$ and hence we have
$$
\supp \P =\{\q \in \spec R \mid \q \not\in \pp(\P) = \p= \pp(\s(\p))\} = \supp \s(\p).
$$
This shows that $\P \subseteq \s(\p)$ and thus $\s(\p)$ is a maximal element in $\spp C$.
For this reason, the map $\s: \min \supp C \to \max \spp C$ is well-defined.

Next we check the surjectivity of the map $\s: \min \supp C \to \max \spp C$.
Let $\P$ be a maximal element of $\spp C$.
Since $\s\pp(\P)$ is also an element in $\spp C$, we get $\P = \s\pp(\P)$ from the maximality of $\P$.
Let $\p$ be an element of $\supp C$ with $\p \subseteq \pp(\P)$.
Then $\P =\s\pp(\P) \subseteq \s(\p)$. 
Since $\P$ is maximal in $\spp C$, one has $\P=\s(\p)$.
Henceforth, $\p = \pp\s(\p)= \pp(\P)$ and this shows that $\pp(\P)$ is a minimal element of $\supp C$.
As a result, $\s(\p)= \s \pp(\P) = \P$ shows that $\s: \min \supp C \to \max \spp C$ is surjective.
\end{proof}

The following result gives an easier way to check whether a given subspace is clopen.

\begin{lem}\label{5}
Let $X$ be either $\supp C$ or $\spp C$ and $W$ a subset of $X$.
If $W$ is both specialization closed and generalization closed, then $W$ is clopen. 
\end{lem}

\begin{proof}
We show this statement only for $X=\spp C$ because the similar argument works also for $X=\supp C$. 
By symmetry, we need to check that $W$ is closed. 

\begin{claim*}
$W= \bigcup_{\P \in \max \spp C} \overline{\{\P\}}$.
\end{claim*}

\begin{proof}[Proof of claim]
Since $W$ is specialization closed, $W \supseteq \bigcup_{\P \in \max \spp C} \overline{\{\P\}}$ holds.
Let $\P$ be an element of $W$.
Take a minimal element $\p$ in $\supp C$ contained in $\pp(\P)$.
We can take such a $\p$ since $\supp C$ is a closed subset of $\spec R$.
Then 
$$
\P \subseteq \s\pp(\P) \subseteq \s(\p).
$$
By Lemma \ref{minmax}, $\s(\p)$ is a maximal element of $\spp C$.
Moreover, $\s(\p)$ belongs to $W$ since $W$ is generalization closed and $\P \in W$.
These show that $\s(\p)$ is a maximal element of $\spp C$. 
Accordingly, we obtain $\P \in \overline{\{\s(\p)\}}$ with $\s(\p) \in \max \spp C$ and hence the converse inclusion holds true.
\renewcommand{\qedsymbol}{$\square$}
\end{proof}

Note that $\supp C$ is closed and thus contains only finitely many minimal elements.
By using the one-to-one correspondence in Lemma \ref{minmax}, $\max \spp C$ is also a finite set.
Consequently, $W$ is a finite union of closed subsets, and hence is closed.
\end{proof}

\begin{lem}\label{sppclopen}
Let $U$ be a clopen subset of $\spp C$. Then
\begin{enumerate}[\rm(1)]
\item $\p \in \pp(U)$ if and only if $\s(\p) \in U$, and
\item $\pp(U)$ is a clopen subset in $\supp C$. 
\end{enumerate}
\end{lem}

\begin{proof}
(1) `if' part is from Proposition \ref{maps}(2).
Let $\p$ be an element of $\pp(U) \subseteq \pp(\spp C) = \supp C$.
Then there is a prime thick tensor ideal $\P \in U$ such that $\pp(\P) = \p$.
Then $\s(\p)$ belongs to $U$ because $\P \subseteq \s \pp(\P)= \s(\p)$ and $U$ is generalization closed in $\supp C$.

(2) 
By Lemma \ref{5}, we have only to check that $\pp(U)$ and $\supp C \setminus \pp(U)$ are specialization closed in $\supp C$.

Take $\p \in \pp(U)$ and $\q \in \V(\p)$.  
Then $\s(\q) \subseteq \s(\p)$.
From $(1)$, one has $\s(\p) \in U$.
Since $U$ is specialization closed, we get $\s(\q) \in U$.
Thus, $\q = \pp\s(\q)$ belongs to $\pp(U)$. 
This shows that $\pp(U)$ is specialization closed in $\supp C$.

Take $\p \in \supp C \setminus \pp(U)$ and $\q \in \V(\p)$. 
Then $\s(\q) \subseteq \s(\p)$.
From $(1)$, one has $\s(\p) \not\in U$.
Assume that $\s(\q)$ belongs to $U$.
Since $U$ is generalization closed, $\s(\p)$ belongs to $U$, a contradiction.
Thus, $\s(\q) \not\in U$ and hence $\q \not\in \pp(U)$ by $(1)$.
This shows that $\supp C \setminus \pp(U)$ is specialization closed in $\supp C$.
\end{proof}

\begin{lem}\label{3}
Let $U$ be a clopen subset of $\spp C$.
Then $\pp^{-1}\pp(U)=U$.
\end{lem}

\begin{proof}
The inclusion $U \subseteq \pp^{-1}\pp(U)$ is trivial.
For a prime thick tensor ideal $\P \in \pp^{-1}\pp(U)$, one has $\pp(\P) \in \pp(U)$.
By Lemma \ref{sppclopen}(1), we obtain $\s\pp(\P) \in U$.
Since $U$ is specialization closed in $\spp C$ and $\P \subseteq \s\pp(\P)$, we conclude that $\P$ belongs $U$.
\end{proof}



Now, we are ready to prove Theorem \ref{main}.

\begin{proof}[(Proof of Theorem \ref{main})]
$(1)$
By Lemma \ref{sppclopen}(2), we obtain a well-defined map
$$
\{\mbox{clopen subsets of } \spp C\} \rightarrow \{\mbox{clopen subsets of } \supp C\}, \,\, U \mapsto \pp(U).
$$
This map is injective by Lemma \ref{3} and surjective since $\pp: \spp C \to \supp C$ is continuous and surjecive.
Thus, this map is an order-preserving one-to-one correspondence.

By definition, connected components are nothing but minimal non-empty clopen subsets. 
Therefore, the statement $(1)$ follows from the above bijection.

$(2)$
By \cite[Proposition 2.9, Proposition 2.18]{B}, every irreducible closed subset of $\spp C$ is of the form 
$$
\overline{\{\P \}} =\{\Q \in \spc \dm(R) \mid \Q \subseteq \P \}
$$
for a unique prime thick tensor ideal $\P \in \spp C$.
Since an irreducible component is by definition a maximal irreducible closed subset, every irreducible component of $\spp C$ is of the form $\overline{\{\P \}}$ for a unique maximal element $\P$ of $\spp C$.
Thus, $\P = \s(\p)$ for some minimal element $\p$ of $\supp C$ by Lemma \ref{minmax}. 
Similarly, every irreducible component of $\supp C$ is of the form $\overline{\{\p\}}$ for a unique minimal element $\p$ of $\supp C$. 
Therefore, there is a maximal element $\P$ of $\spp C$ such that $\p=\pp(\P)$ by Lemma \ref{minmax}.
Altogether, the one-to-one correspondence of Lemma \ref{minmax} gives a one-to-one correspondence what we want.
\end{proof}

The following connectedness result is a direct consequence of Theorem \ref{main}.

\begin{cor}\label{conn}
For a bounded complex $C \in \db(R)$, $\spp C$ is connected (resp. irreducible) if and only if $\supp C$ is connected (resp. irreducible).
In particular, $\spc \dm(R)$ is connected (resp. irreducible) if and only if $\spec R$ is connected (resp. irreducible).
\end{cor}

\begin{rem}
A part of this corollary is shown in \cite[Corollary 4.13]{MT}.
\end{rem}




From now, let me mention two applications of Theorem \ref{main}.
We start with the following lemma.

\begin{lem}\label{supcon}
Let $C$ be a finitely generated $R$-module.
If $C$ is indecomposable, then $\supp C$ is connected.
\end{lem}

\begin{proof}
Set $I:= \ann C$. 
Consider a decomposition $\supp C = F_1 \sqcup F_2$ with $F_1, F_2$ closed.
Then there are radical ideals $I_1$ and $I_2$ such that $I_1 + I_2 = R$ and $I_1 \cap I_2 = \sqrt{I}$.
Using Chinese remainder theorem, we obtain a direct sum decomposition
$$
R/\sqrt{I} \cong R/I_1 \oplus R/I_2. 
$$
Moreover, from the idempotent lifting theorem (see \cite[Proposition 21.25]{L}), we obtain the following decomposition
$$
R/ I \cong R/J_1 \oplus R/J_2.
$$
Here, $J_1$ and $J_2$ are ideals with $\sqrt{J_i}=I_i$ for $i=1, 2$.
Tensoring with $C$, we obtain the following direct sum decomposition:
$$
C \cong C \otimes_R R/I \cong (C \otimes_R R/J_1) \oplus (C \otimes_R R/J_2).
$$
Since $C$ is indecomposable, $C \otimes_R R/J_1 \cong C$ or $C \otimes_R R/J_2 \cong C$.
If $C \otimes_R R/J_1 \cong C$, then $J_1 \subseteq \ann C=I$ and thus $I=J_1$.
In this case, $F_1=\supp C$. 
Similarly, if $C \otimes_R R/J_2 \cong C$, then one has $F_2= \supp C$.
Thus, we are done.
\end{proof}

The following corollary is a direct consequence of Theorem \ref{main} and Lemma \ref{supcon}.

\begin{cor}
Let $C$ be an indecomposable finitely generated $R$-module.
Then $\spp C$ is connected. 
\end{cor}

This corollary means the classifying support, in the sense of Balmer, of indecomposable $R$-module is connected.
Such a result has been shown by Carlson \cite{C} for the stable category of finite dimensional representations over a finite group, and more generally, by Balmer \cite{B2} for an idempotent complete strongly closed tensor triangulated category.

Next, we prove that every clopen subset of $\spc \dm(R)$ is homeomorphic to 
the Balmer spectrum of the Eilenberg-Moore category of some ring object.
Following \cite{B3, B4}, we recall the notion of a ring object and related concepts. 

Let $(\T, \otimes, \one)$ be a tensor triangulated category.
We say that an object $A \in \T$ is a {\it ring object} of $\T$ if there is a morphisms
\begin{align*}
&\mu: A \otimes A \to A, \\
&\eta: \one \to A \\
\end{align*}
satisfying the following commutative diagrams:
$$
\xymatrix{
A \otimes A \otimes A \ar[r]^{A \otimes \mu} \ar[d]_{\mu \otimes A} & A \otimes A \ar[d]^\mu\\
A \otimes A \ar[r]_\mu & A
}\,\,\,\,\,\,\,\,\
\xymatrix{
\one \otimes A \ar[r]^{\eta \otimes A} \ar[rd]_{\cong} & A \otimes A \ar[d]_\mu & A \otimes \one \ar[l]_{A \otimes \eta} \ar[ld]^{\cong}\\
& A &
}
$$
We say that a ring object $A$ of $\T$ is {\it commutative} if there is a morphism
$$
\tau: A \otimes A \to A \otimes A
$$
such that $\mu \tau = \mu$.
We say that a ring object $A$ of $\T$ is {\it separable} if there is a morphism
$$
\sigma: A \to A \otimes A
$$
such that $(A \otimes \mu)(\sigma \otimes A) = \sigma \mu = (\mu \otimes A)(A \otimes \sigma)$.

We say that an object $M \in \dm(R)$ is a (left) $A$-{\it module} if there is a morphism
$$
\lambda: A \otimes M \to M
$$
satisfying the following commutative diagrams:
$$
\xymatrix{
A \otimes A \otimes M \ar[r]^{A \otimes \lambda} \ar[d]_{\mu \otimes M} & A \otimes M \ar[d]^\lambda \\
A \otimes M \ar[r]_\lambda & M
}\,\,\,\,\,\,\,\,\
\xymatrix{
\one \otimes M \ar[r]^{\eta \otimes M} \ar[rd]_{\cong} & A \otimes M \ar[d]^\lambda \\
& M 
}
$$
Denote by $\Mod A$ the category of $A$-modules.
For an object $X$ of $\T$, $A \otimes X$ has an $A$-module structure via
$$
\mu \otimes X : A \otimes A \otimes X \to A \otimes X 
$$ 
Thus, we have a functor $F_A: \T \to \Mod A$.

Balmer \cite{B3} shows that if a ring object $A$ is separable, then the category $\Mod A$ admits a unique triangulated category structure  such that both $F_A: \T \to \Mod A$ and the forgetful functor $U_A: \Mod A \to \T$ are exact.
Moreover, if $A$ is commutative, then $\Mod A$ has a symmetric monoidal tensor product $\otimes_A$ and it makes $\Mod A$ the tensor triangulated category such that $F_A$ is a tensor triangulated functor.

Let me give the following easy observation.

\begin{lem}\label{ringobj}
If $R$ is decomposed into $R = A \times B$ as rings, then $A$ has a unique ring object structure by the natural multiplication $\mu: A \ltensor_R A \cong A \otimes_R A \cong A$ and the projection $\eta: R \to A$.
Moreover, the following holds true.
\begin{enumerate}[\rm(1)]
\item
$A$ is a commutative separable ring object in $\dm(R)$.
\item
For any complex $X \in \dm(R)$, it has an $A$-module structure if and only if $A \ltensor_R X \cong X$.
This is the case, its $A$-module structure is uniquely determined by underlying complex.
\item
$U_A$ preserves tensor products.
Namely, for $A$-modules $M, N$, one has $M \otimes_A N \cong M \ltensor_R N$ in $\dm(R)$.
\end{enumerate} 
\end{lem}

\begin{proof}
Since $A$ is a projective $R$-module, the statement $(1)$ means that $A$ is a commutative separable $R$-algebra in the usual sense and this is clear.
Uniqueness of this structure follows from $(2)$.

$(2)$ Let $M$ be an $A$-module. 
Consider the following commutative diagram:
$$
\xymatrix{
R \ltensor_R X \ar[r]^{\eta \ltensor_R X} \ar[rd]_{\cong} & A \ltensor_R X \ar[d]^\lambda \\
& M 
}
$$
Since $\eta \ltensor_R M$ is a split epimorphism, it must be an isomorphism and hence $\lambda$ is also an isomorphism.
The converse is trivial.
Moreover, the $A$-module structure $\lambda$ is uniquely determined as
$$
A \ltensor_R X \xrightarrow{(\eta \ltensor_R X)^{-1}} R \ltensor_R X \xrightarrow{\cong} X.
$$ 

The last statement $(3)$ directly follows from the definition of $\otimes_A$ and $(2)$, for details, see \cite{B4}.
\end{proof}

From $(2)$ in the above lemma, we can define a unique $A$-module structure for a complex $X \in \dm(R)$ with $A \ltensor_R X \cong X$.
For simplicity, we denote this $A$-module by $X_A$.
\begin{cor}
For any non-empty clopen subset $W$ of $\spc \dm(R)$, there is a commutative separable ring object $A$ of $\dm(R)$ such that 
$$
\varphi_A:={}^aF_A: \spc (\Mod A) \to \spc \dm(R),\,\, \P \to F_A^{-1}(\P)
$$
gives a homeomorphism onto $W$.
\end{cor}

\begin{proof}
By Lemma \ref{sppclopen}, $\pp(W)$ is a clopen subset of $\spec R$.
Therefore, by using the argument in the proof of Lemma \ref{supcon}, there is a direct sum decomposition $R = A \times B$ of rings with $\pp(W) = \supp A$. 
Then Lemma \ref{ringobj} shows that $A$ has a commutative separable ring object structure.
Since $U_A$ preserving tensor products, one can easily check that the forgetful functor $U_A: \Mod A \to \dm(R)$ induces a continuous injective map
$$
\psi_A:\spp A \to \spc(\Mod A), \P \to U_A^{-1}(\P),
$$
see \cite[Proposition 3.6]{B}.
Furthermore, the image of $\varphi_A$ is contained in $\spp A$ and $\psi_A\varphi_A=1$ because $F_A U_A \cong 1$. 
For this reason, we have only to check that the image of $\varphi_A$ is $W$.

Let $\P$ be a prime thick tensor ideal of $\Mod A$.
By definition, 
$$
\varphi_A(\P)=\{X \in \dm(R) \mid F_A(X)=(A \ltensor_R X)_A \in \P\}
$$
and it contains $B$ because $A \ltensor_R B =0$.
In particular, 
$$
\supp B \subseteq \supp \varphi_A(\P) =\{\p \in \spec R \mid \p \not\subseteq \pp(\varphi_A(\P))\}
$$ 
and thus $\pp(\varphi_A(\P)) \in \spec R \setminus \supp B = \supp A$.
Therefore, $\varphi_A(\P) \in W$ by Lemma \ref{3}.
Conversely, take a prime thick tensor ideal $\P$ from $W$.
Then $\pp(\P) \in \pp(W) = \supp A$ implies that $A \not\in \P$.
Therefore, 
$$
\varphi_A(\psi_A(\P))=\{X \in \dm(R) \mid A \ltensor_R X \in \P\}=\P
$$
since $A \notin \P$.
Thus, we conclude that $\varphi_A(\spc(\Mod A))=W$. 
\end{proof}



\begin{ac} 
The author is grateful to his supervisor Ryo Takahashi for his many grateful comments.
\end{ac}

\end{document}